\documentclass[reqno, 12pt]{amsart}

\usepackage{mathrsfs}
\usepackage{amscd}
\usepackage{amsmath}
\usepackage{latexsym}
\usepackage{amsfonts}
\usepackage{amssymb}
\usepackage{amsthm}
\usepackage{graphicx}
\usepackage{hyperref}
\usepackage{makecell}
\usepackage{array}
\usepackage{booktabs}
\usepackage{multirow}
\usepackage{color,xcolor}
\usepackage{bm}

\newtheorem{theorem}{Theorem}[section]
\newtheorem{proposition}[theorem]{Proposition}

\newtheorem{lemma}[theorem]{Lemma}
\newtheorem{definition}[theorem]{Definition}

\newtheorem{assumption}[theorem]{Assumption}

\newcommand\D{\mathcal{D}}

\numberwithin{equation}{section}

\title[PML method for the stochastic acoustic scattering problem]{PML method for the stochastic acoustic scattering problem driven by an additive Gaussian noise}

  \author[H. Guo]{Hongxia Guo}
 \address{School of Mathematical Sciences, and Institute of Mathematics and Interdisciplinary Sciences, Tianjin Normal University, Tianjin, 300387, China}
 \email{hxguo@tjnu.edu.cn}

  {\author[T. Wang]{Tianjiao Wang}\address{School of Mathematical Sciences, Zhejiang University, Hangzhou 310058, China}}\email{wangtianjiao@zju.edu.cn}

\author[X. Xu]{Xiang Xu}
\address{School of Mathematical Sciences, Zhejiang University, Hangzhou 310058, China}
\email{xxu@zju.edu.cn}

 \author[Y. Zhao]{Yue Zhao}
\address{School of Mathematics and Statistics, and Key Lab NAA-MOE, Central China Normal University,
Wuhan 430079, China}
\email{zhaoy@ccnu.edu.cn}


\begin{document}

\begin{abstract}

This paper is concerned with the time-domain stochastic acoustic scattering problem driven by a spatially white additive Gaussian noise.
The main contributions of the work are twofold. First, we prove the existence and uniqueness of the pathwise solution to the scattering problem by applying an abstract Laplace transform inversion theorem. The analysis employs the black box scattering theory to investigate the meromorphic continuation of the Helmholtz resolvent defined on rough fields. Second, based on the piecewise constant approximation of the white noise, we construct an approximate wave solution and establish the error estimate. As a consequence,
we develop a PML method and establish the convergence analysis with explicit dependence on the PML layer's thickness and medium properties, as well as the piecewise constant approximation of the white noise. 

\end{abstract}

\maketitle

\section{Introduction}

Stochastic partial differential equations are known to be effective tools in modeling complex physical and engineering phenomena including the wave propagation
\cite{bao, Papan}.
In this paper, we focus on the wave propagation governed by the three-dimensional stochastic acoustic wave equation with sources driven by spatially uncorrelated additive Gaussian noise.  Gaussian noise sources which are spatially uncorrelated have been proposed and investigated in \cite{garnier, zhou} for wave propagations.

Specifically, we consider the stochastic acoustic scattering problem with the Dirichlet boundary condition on the obstacle
\begin{align}\label{eqn}
\begin{cases}
\partial^2_t u(x, t)-\Delta u(x, t) =  f(x, t)\, \dot{W}_x, \quad &(x, t)\in(\mathbb R^3\setminus\overline{O})\times [0, \infty),\\
u = 0,   \quad & (x, t)\in\partial O\times  [0, \infty),\\
u(x, 0) = \partial_tu(x, 0)=0, \quad &x\in\mathbb R^3\setminus \overline{O},
\end{cases}
\end{align}
where $O$ is a bounded domain with smooth boundary $\partial O$ such that $ O \subset \subset B_R:= \{x \in \mathbb{R}^3 : |x| < R\}$ with $R>0$.
Here $\dot{W}_x$ denotes spatial white noise and the function $f \in C^\infty(\mathbb{R}^3\times \mathbb{R}_+)$ characterizes the strength of the random source. 
We assume that $f(x, t)$ is compactly supported in both space and time with ${\rm supp}f\subset\subset (B_R\setminus\overline{O}) \times (0, T)$ where $T > 0$
is a positive constant.

The analysis and computation of the time-domain wave scattering problems have been extensively investigated in deterministic settings. Well-posedness of the acoustic
and electromagnetic scattering problems have been established in \cite{chen, lv, Wei20}. The proofs depend on the abstract inversion theorem of the Laplace transform and the a priori estimate for the Helmholtz and Maxwell equations in the frequency domain. Also, numerical methods have been developed for solving wave scattering problems. 
Initiated by B$\acute{\rm e}$renger \cite{Ber94} for time-domain Maxwell equations, the perfectly matched layers (PML) method has gained widespread popularity due to its near-perfect absorption of outgoing waves \cite{hoop, Joly, Joly12}. Assuming that the wave is outgoing,
 the basic idea of the PML is to surround the physically computational domain by some artificial medium that absorbs the waves effectively.
As another commonly used alternative, we refer to the method of absorbing boundary conditions \cite{Hag, Tsy}.

Compared with deterministic partial differential equations, the difficulty in the study of a SPDE is the lack of regularity of its solution.
For example, the required regularity is not satisfied for the problem \eqref{eqn} to use variational methods as in \cite{chen, lv} to derive the a priori estimate
and establish the analyticity for the solution to the corresponding Helmholtz equation. To fix this issue, we employ the black box scattering theory in \cite{Dyatlov} to investigate
the meromorphic continuation of the resolvent of the Helmholtz equation with respect to complex wave number, which yields the well-posedness of the time-domain scattering problem \eqref{eqn}
via the abstract inversion theorem of the Laplace transform.

After establishing its well-posedness, we propose a PML method for computing the time-domain stochastic scattering problem. The main challenge is also the low regularity
of the solution. To overcome this difficulty, we consider replacing $\dot{W}_x$ with its discretized piecewise constant approximation as in \cite{Zhang, du}.
As a consequence, the solution of the resulting SPDE has the desired regularity for applying the PML method. 
Then we derive an error estimate in $L^2$-norm for the approximation of \eqref{eqn} with discretized white noises.
Based on the error estimate, we establish the convergence analysis
of the PML method with explicit dependence on the PML layer's thickness and medium properties, as well as the piecewise constant approximation of the white noise.

This paper is organized as follows. In Section \ref{wp}, we prove the pathwise existence and uniqueness of the solution for the stochastic scattering problem.
In Section \ref{pml}, we study the approximation of the equation \eqref{eqn} using discretized white noises and establish the error
estimate in  $L^2$-norm. As a consequence, we propose the PML method and establish its convergence analysis.

\section{Well-posedness of the stochastic scattering problem}\label{wp}

In this section, we establish the pathwise well-posedness and the stability of the time-dependent stochastic acoustic scattering problem. 
The analysis is based on the examination of the time-harmonic stochastic Helmholtz equation and the abstract inversion theorem of the Laplace transform. 

For any complex number \( s = s_1 + \mathrm{i}s_2 \in \mathbb{C} \) with \( \Re s > 0 \), the Laplace transform of \( u \) with respect to time $t$ is defined by  
\[
u_L(x, s) := \mathcal{L}(u)(x, s) = \int_0^\infty e^{-st} u(x, t) \, \mathrm{d}t.
\]

Using the Laplace transform in time and the identity  
$\mathcal{L}(\partial_t^2 u)(x, s) = s^2 u_L(x, s) - s u(x, 0) - \partial_t u(x, 0),$
we transform the acoustic wave equation \eqref{eqn} into the inhomogeneous Helmholtz equation  
\begin{align} \label{helm}
\begin{cases}
- \Delta u_L + s^2 u_L = f_L(x, s)\, \dot{W}_x,&\quad x\in\mathbb R^3,\\
u_L=0, &\quad x\in\partial O,
\end{cases}
\end{align}
where \( u_L \) satisfies the  radiation condition  
\[
r \left( \partial_r u_L + s u_L \right) \to 0 \quad \text{as } r \to \infty, \quad r=|x|.
\]
Our strategy to establish the well-posedness of the time-domain stochastic scattering problem \eqref{eqn} is to show the inverse Laplace transform of the solution
$u_L$ is existent. To this end, we employ the following lemma (cf. \cite[Theorem 43.1]{Treves}) concerning the inverse Laplace transform.
\begin{lemma} \label{lem:a}
Let $\omega_L(s)$ denote a holomorphic function in the half complex plane $\Re (s)>\sigma_0$ for some $\sigma_0\in\mathbb R$, valued in the Banach space $\mathcal E$. The following conditions are equivelent:
\begin{itemize}
    \item[(1)] there is a distribution $\omega \in \mathcal{D}_+^{'}(\mathcal E)$ whose Laplace transform is equal to  $\omega_L(s)$, where $\mathcal{D}^{'}_+$ is the space of distributions on the real line which vanish identically in the open negative half-line;
    \item[(2)] there is a $\sigma_1$ with $\sigma_0\leq\sigma_1<\infty$ and an integer $m\geq0$ such that for all complex numbers $s$ with $\textnormal{Re} (s)>\sigma_1$  it holds that $\|\omega_L(s)\|_{\mathcal E}\leq C(1+|s|)^m$.
\end{itemize}
\end{lemma}

In light of Lemma \ref{lem:a}, we shall investigate the analyticity of $u_L$ with respect to complex wave number in order to apply the inverse Laplace transform.
To this end, in what follows,
we employ the black box scattering theory (cf. \cite[Section 4]{Dyatlov}) in the context of the obstacle scattering problem in the frequency domain, which will yield the analyticity and estimates of $u_L$ with respect to complex wave number. 

 We introduce the unbounded operator $P=-\Delta: \D(P) \to L^2(\mathbb R^3\backslash \overline{O})$ with domain $\D(P):=\{u \in H^2(\mathbb R^3 \backslash \overline O): u|_{\partial O}=0\}$. It is well known that the spectrum of $P$ is contained in $[0, +\infty)$. Therefore,  the resolvent $R(\lambda):=(P-\lambda^2)^{-1}: L^2(\mathbb R^3\backslash \overline{O}) \to \D(P)$ is well-defined when $\Im \lambda >0.$ However, the resolvent $R(\lambda)$ itself can not be meromorphically extended from $\Im  \lambda>0$ to $\mathbb C$. Fortunately, multiplied by a cut-off function $\chi\in C_0^\infty(\mathbb R^3\setminus\overline{O})$ on both sides, the resolvent
 $\chi R(\lambda) \chi$ can be meromorphically extended from $\Im  \lambda>0$ to $\mathbb C$ as a family of operators: $L^2(\mathbb R^3\backslash \overline{O}) \to\mathcal D(P)$  
 with poles in $\{\lambda:\Im \lambda <0\}$ (cf. \cite[Theorem 4.4]{Dyatlov}). 

 In order to obtain the resolvent estimates and analytic domain, we need the following non-trapping condition imposed on $P$ (cf. \cite[Definition 4.42]{Dyatlov}). Notice that $P$ is self-adjoint so that we can use the spectral theorem to define \[U(t)=\frac{\sin(t\sqrt P)}{\sqrt{P}}.\]
 \begin{assumption}\label{ntrap}
For any $a>R$, $\exists \, T_a>0$ such that for each $\chi\in C_0^\infty
(B_a), \chi|_{B_{R+\tau}}=1$ with $\tau$  being a small constant, it holds that 
\[
\chi U(t)\chi|_{t>T_a} \in C^\infty ((T_a, \infty);\mathcal
L(L^2(\mathbb R^3 \backslash \overline O), \D(P))).
\] 
\end{assumption}

The following proposition \cite[Theorem 4.43]{Dyatlov} concerns the analyticity and estimate of the resolvent $R(\lambda)$ defined on $L^2$ functions with compact supports.
\begin{proposition}\label{nte}
If $P$ is non-trapping under Assumption \ref{ntrap}, then
for any $M>0$, there exists $C_0>0$ such that $R(\lambda):L_{comp}^2(\mathbb R^3 \backslash \overline{O})\to L_{loc}^2(\mathbb R^3\backslash \overline{O})$ is holomorphic
as a family of operators for $\lambda$
in the domain
\[
\Omega_M=\{\lambda\in\mathbb{C}:\,\mathrm{Im}\lambda\geq -M\log|\lambda|,\quad |\lambda|\geq C_0\}.
\]
Assuming $\chi\in C_c^\infty(\mathbb{R}^3)$ such that $\chi=1$ near $B_R$ (the black box), the
following resolvent estimate holds:
\begin{equation}\label{stability_resolvent}
\|\chi R(\lambda)\chi\|_{L^2(\mathbb R^3 \backslash \overline{O})\to H^j(\mathbb R^3\backslash \overline{O})}\leq
C|\lambda|^{j-1}e^{L({\rm Im}\lambda)_-}, \quad j=0, 1, 2,
\end{equation}
where $\lambda\in \Omega_M$,  $C, L$ are positive constants and $({\rm
Im}\lambda)_-:=\max(0,-{\rm Im}\lambda)$. 
\end{proposition}

Denote
the free resolvent 
\[
{R}_0(\lambda):=(-\Delta-\lambda^2)^{-1}
\] which has the explicit kernel
\[
G_0(x, y; \lambda)=\frac{e^{{\rm i}\lambda|x-y|}}{4\pi |x-y| }.
\]
Notice that the above proposition can not be applied directly in the following analysis since the white noise is too rough to be in $L^2$. To resolve this issue, we begin by establishing the analyticity and estimates of the free resolvent $\chi R_0(\lambda)\chi$  as a family of operators from $H^\gamma(\mathbb R^3 \backslash \overline O)$ to $H^\gamma(\mathbb R^3 \backslash \overline O)$ for any real number $\gamma\in\mathbb R$.
Before that, we recall the following definition of outgoing functions (cf. \cite[Definition 1.6]{notes}). \begin{definition}
    Given $\lambda \in \mathbb C$, we say that the function $u$ is outgoing, if there exists $A>0$ and $f \in \mathcal E'$ such that $u|_{|x|>A}=
    (R_0(\lambda)f)|_{|x|>A}.$
\end{definition}

The following lemma  (cf. \cite[Lemma 3.3]{WXZ}) is useful in the subsequent analysis, which concerns the analyticity and estimates of the free resolvent. 

	 \begin{lemma}\label{Lem4.1}
	 Let $\chi \in C_0^\infty(\mathbb{R}^3)$ be a cut-off function.
	The free resolvent $R_0(\lambda)$ is analytic for $\lambda \in \mathbb{C}$ with $\Im \lambda>0$ as a family of operators \[
 R_0(\lambda): H^\gamma(\mathbb{R}^3) \to H^{\gamma_1}(\mathbb{R}^3)
 \] for all $\gamma \in \mathbb{R}$ and $\gamma_1\in [\gamma, \gamma+2]$. Furthermore, ${R}_0(\lambda)$ can be extended to a family of analytic operators for $\lambda \in \mathbb{C}$ as follows \[
 \chi R_0(\lambda) \chi : H^\gamma(\mathbb{R}^3) \to H^{\gamma_1}(\mathbb{R}^3)
 \] with the resolvent estimates
 \begin{equation*}
     \|\chi R_0(\lambda) \chi \|_{\mathcal{L}(H^\gamma,H^{\gamma_1})} \leq C(1+|\lambda|)^{\gamma_1-\gamma-1}e^{L (\Im \lambda)_-},
 \end{equation*}
 where $C$ is a positive constant, the constant $L$ is only dependent on $supp \, \chi$, and $\|\cdot\|_{\mathcal{L}(H^\gamma,H^{\gamma_1})}$ denotes the operator norm from $H^\gamma$ 
 to $H^{\gamma_1}$.
  \end{lemma}

The following proposition concerns the analyticity and estimates of the resolvent $R(\lambda)$, which is crucial in applying the abstract inverse Laplace transform.
It is also of independent interests in scattering theory.
\begin{proposition}\label{ntes}
Assume $\chi_1\in C_c^\infty(\mathbb{R}^3)$ and $\chi_2 \in C_c^\infty(\mathbb{R}^3\backslash \overline{O})$ such that $\chi_1=1$ near ${\rm supp}\,\chi_2$. 
For any $M>0$, there exists $C_0>0$ such that $\chi_1 R(\lambda) \chi_2: H^{\gamma}(\mathbb R^3 \backslash \overline O) \to H^{\gamma+j}(\mathbb R^3 \backslash \overline O), j = 0, 1, 2,$ is holomorphic
in the domain
\[
\Omega_M=\{\lambda\in\mathbb{C}:\,\mathrm{Im}\lambda\geq -M\log|\lambda|,\quad |\lambda|\geq C_0\}.
\]
Moreover, the
following resolvent estimates hold:
\begin{equation*}
\|\chi_1 R(\lambda) \chi_2  \|_{H^{\gamma}(\mathbb R^3 \backslash \overline O) \to H^{\gamma+j}(\mathbb R^3 \backslash \overline O)} \le C e^{L(\Im \lambda)_-}|\lambda|^{|\gamma|+1+ j}, \quad j=0, 1,2.
\end{equation*}
for {$\gamma \in \mathbb 
R$} and $\lambda\in \Omega_M$, where $C, L$ are constants and $({\rm
Im}\lambda)_-:=\max\{0,-{\rm Im}\lambda\}$.
\end{proposition}
\begin{proof}
  We begin by investigating the Schwarz kernel of the resolvent $R(\lambda)$. Suppose that $G_D(x,y;\lambda)$ is the Dirichlet Green function of the obstacle scattering problem. In other words, for any fixed $y \in \mathbb R^3 \backslash \overline O$, $G_D(\cdot,y;\lambda)$ is the unique solution to the Dirichlet problem \begin{align}\label{gf1}
        \begin{cases}
            (\Delta+\lambda^2)G_D(\cdot, y; \lambda)=-\delta_y, \\
            G_D(\cdot,y;\lambda)=0\quad \text{\rm on}\quad \partial O, \\
            G_D(\cdot,y;\lambda) \,\textrm{is outgoing}.
        \end{cases}
    \end{align}
   We decompose the Dirichlet Green function into the sum of the fundamental solution and a non-singular part as $G_D(x,y;\lambda)=G_0(x,y;\lambda)+P(x,y;\lambda)$. Then \eqref{gf1} is equivalent to the following problem \begin{align}\label{gf2}
        \begin{cases}
            (\Delta+\lambda^2)P(\cdot,y; \lambda)=0\quad \text{\rm in}\quad  \mathbb R^3 \backslash \overline O, \\
            P(\cdot,y;\lambda)=-G_0(\cdot,y;\lambda)\quad \text{\rm on}\quad \partial O,\\
           P(\cdot,y;\lambda) \,\textrm{is outgoing}.
        \end{cases}
    \end{align} Extend the smooth function $G_0(\cdot, y; \lambda)|_{\partial O}$ to $\widetilde G_0(\cdot, y; \lambda) \in C_c^\infty(\mathbb R^3\backslash O)$ and let $\widetilde P=P+\widetilde G_0$. As a result, \eqref{gf2} is equivalent to \begin{align}\label{gf3}
        \begin{cases}
            (\Delta+\lambda^2)\widetilde P(\cdot, y; \lambda)=(\Delta+\lambda^2)\widetilde G_0(\cdot, y; \lambda) \quad \text{\rm in}\quad  \mathbb R^3 \backslash \overline O, \\
           \widetilde P(\cdot,y;\lambda)=0\quad \text{\rm on}\quad \partial O, \\
            \widetilde P(\cdot,y;\lambda) \,\textrm{is outgoing}.
        \end{cases}
    \end{align}
  By \cite[Proposition 5.2]{notes}, the scattering problem \eqref{gf3} admits a unique solution for $\lambda \in \Omega_M$, and thus there also exists a unique solution to \eqref{gf1}. Hence, the Green function $G_D$ is well-defined. Furthermore, it can be verified for any $g \in L^2_{comp}(\mathbb R^3\backslash \overline O)$, the function $\int_{\mathbb R^3\backslash \overline{O}} G_D(x,y;\lambda) g(y)\,\mathrm{d}y$ satisfies \begin{align}\label{exp}
        \begin{cases}
            (\Delta+\lambda^2)u=-g, \\
           u=0\quad \text{\rm on}\quad \partial O, \\
           u \,\,\textrm{is outgoing}.
        \end{cases}
    \end{align}
    On the other hand, applying \cite[Proposition 5.2]{notes} again,  the problem \eqref{exp} admits a unique solution $R(\lambda)g\in H_{loc}^2(\mathbb R^3 \backslash \overline O)$ for $\lambda \in \Omega_M.$ 
    As a consequence, we arrive at \[R(\lambda)g=\int_{\mathbb R^3\backslash \overline{O}} G_D(x,y;\lambda) g(y)\,\mathrm{d}y=\int_{\mathbb R^3\backslash \overline{O}}(G_0(x,y;\lambda)+P(x,y;\lambda))g(y)\,\mathrm{d}y.\]
    
    We consider the resolvent estimates.
    The case $\gamma = 0$ is a direct consequence of Proposition \ref{nte}. 
    Now we derive the estimate for $\gamma=1.$ Suppose that $q \in H^1(\mathbb R^3\backslash \overline{O})$. Then 
    \[\chi_1 R(\lambda) (\chi_2 q)=\chi_1R_0(\lambda)(\chi_2 q)+\int_{\mathbb R^3\backslash \overline{O}}\chi_1(\cdot)P(\cdot,y;\lambda)\chi_2(y)q(y)\,\mathrm{d}y.\]
    The estimate for the first part can be directly obtained from Lemma \ref{Lem4.1}. Consequently, it remains to estimate the second part. Taking derivative gives \begin{align*}
    &\partial^\alpha\int_{\mathbb R^3\backslash \overline{O}}\chi_1(\cdot)P(\cdot,y;\lambda)\chi_2(y)q(y)\,\mathrm{d}y \\ & \quad=\int_{\mathbb R^3\backslash \overline{O}}\partial^\alpha\chi_1(\cdot)P(\cdot,y;\lambda)\chi_2(y)q(y)\,\mathrm{d}y+\int_{\mathbb R^3\backslash \overline{O}}\chi_1(\cdot)\partial^\alpha P(\cdot,y;\lambda)\chi_2(y)q(y)\,\mathrm{d}y,
    \end{align*} where $\alpha\in (\mathbb N^*)^3$ with $|\alpha|=1$. Here $\mathbb N^*$ denotes the set of positive integers.
    By Minkovski inequality and Cauchy-Schwarz inequality, there holds the estimate \begin{align*}
   & \left\|\int_{\mathbb R^3\backslash \overline{O}}\chi_1(\cdot) P(\cdot,y;\lambda)\chi_2(y)q(y)\,\mathrm{d}y\right\|_{H^1(\mathbb R^3\backslash \overline{O})} \\ 
   &= \left\|\int_{\mathbb R^3\backslash \overline{O}}\chi_1(\cdot)P(\cdot,y;\lambda)\chi_2(y)q(y)\,\mathrm{d}y\right\|_{L^2(\mathbb R^3\backslash \overline{O})} \\ 
   &\quad+\sum_{j=1}^3\left\|\int_{\mathbb R^3\backslash \overline{O}}\partial_j\chi_1(\cdot)P(\cdot,y;\lambda)\chi_2(y)q(y)\,\mathrm{d}y\right\|_{L^2(\mathbb R^3\backslash \overline{O})} \\ &\quad+\sum_{j=1}^3\left\|\int_{\mathbb R^3\backslash \overline{O}}\chi_1(\cdot)\partial_{j} P(\cdot,y;\lambda)\chi_2(y)q(y)\,\mathrm{d}y\right\|_{L^2(\mathbb R^3\backslash \overline{O})} \\ 
   &\le C\|\chi_1(x)\chi_2(y)P(x,y;\lambda)\|_{L^2(\mathbb R^3\backslash \overline{O} \times \mathbb R^3\backslash \overline{O})}\|q\|_{L^2(B_{R'}\backslash \overline{O})} \\ &\quad+C\sum_{j=1}^3\|\partial_{x_j}\chi_1(x)\chi_2(y)P(x,y;\lambda)\|_{L^2(\mathbb R^3\backslash \overline{O} \times \mathbb R^3\backslash \overline{O})}\|q\|_{L^2(B_{R'}\backslash \overline{O})}\\ &\quad+C\sum_{j=1}^3\|\chi_1(x)\chi_2(y)\partial_{x_j}P(x,y;\lambda)\|_{L^2(\mathbb R^3\backslash \overline{O} \times \mathbb R^3\backslash \overline{O})}\|q\|_{L^2(B_{R'}\backslash \overline{O})} \\ 
   &\le C\left(\|P(x,y;\lambda)\|_{L^2(B_{R'}\backslash \overline{O} \times F)}+\sum_{j=1}^3\|\partial_{x_j}P(x,y;\lambda)\|_{L^2(B_{R'}\backslash \overline{O} \times F)}\right)\|q\|_{L^2(B_{R'}\backslash \overline{O})},
    \end{align*} 
    where $R'$ is a constant such that ${\rm supp }\, \chi_2 \subset {\rm supp }\, \chi_1 \subset B_{R'}$ and $F$ is a domain such that $ {\rm supp }\, \chi_2  \subset F \subset\subset  B_{R'}\backslash \overline O$. We need to estimate $\|\partial_x^\alpha P(x,y;\lambda)\|_{L^2(B_{R'}\backslash \overline{O} \times F)}$ for $|\alpha| \le 1.$ Recalling \eqref{gf3}, by applying Proposition \ref{nte} we obtain  \begin{align*}
    \|\widetilde P(\cdot,y;\lambda)\|_{H^1(B_{R'}\backslash \overline{O})} &\le C e^{L(\Im \lambda)_-}\|(\Delta+\lambda^2)\widetilde G_0(\cdot ,y;\lambda)\|_{L^2(\mathbb R^3\backslash \overline{O})} \\ &\le Ce^{L(\Im \lambda)_-}\|\widetilde G_0(\cdot ,y;\lambda)\|_{H^2(\mathbb R^3\backslash \overline{O})}+C |\lambda|^2e^{L(\Im \lambda)_-}\|\widetilde G_0(\cdot ,y;\lambda)\|_{L^2(\mathbb R^3\backslash \overline{O})}. 
    \end{align*} Applying the inverse trace theorem, we can choose the extension $\widetilde G_0(\cdot,y;\lambda)$ of $G_0(\cdot,y;\lambda)|_{\partial O}$ such that \[
    \|\widetilde G_0(\cdot ,y;\lambda)\|_{H^2(\mathbb R^3\backslash \overline{O})} \le C\| G_0(\cdot ,y;\lambda)\|_{H^{3/2}(\partial{O})},\, \|\widetilde G_0(\cdot ,y;\lambda)\|_{H^{1/2}(\mathbb R^3\backslash  \overline{O})} \le C\| G_0(\cdot ,y;\lambda)\|_{L^2(\partial{O})}.
    \] Hence, there holds 
    \begin{align*}
   \|P(\cdot,y;\lambda)\|_{H^1(B_{R'}\backslash \overline{O})} &\le \|\widetilde P(\cdot,y;\lambda)\|_{H^1(B_{R'}\backslash \overline{O})}+\|\widetilde G_0(\cdot,y;\lambda)\|_{H^1(B_{R'}\backslash \overline{O})} \\ &\le C  e^{L(\Im \lambda)_-}\| G_0(\cdot ,y;\lambda)\|_{H^{3/2}(\partial{O})} +C|\lambda|^2e^{L(\Im \lambda)_-}\| G_0(\cdot ,y;\lambda)\|_{L^{2}(\partial{O})},
    \end{align*} which gives \begin{align*}
    \int_{D}\|P(\cdot,y;\lambda)\|^2_{H^1(B_{R'}\backslash \overline{O})}\,\mathrm{d}y &\le C^2 e^{2L(\Im \lambda)_-}\int_{D}\|\widetilde G_0(\cdot ,y;\lambda)\|^2_{H^{3/2}(\partial{O})}\,\mathrm{d}y \\ &\quad+C^2|\lambda|^4e^{2L(\Im \lambda)_-}\int_D\|\widetilde G_0(\cdot ,y;\lambda)\|^2_{L^2(\partial{O})}\,\mathrm{d}y\\ & \le C^2 |\lambda|^{4}e^{2L(\Im \lambda)_-}.
    \end{align*}
    In summary, we have obtained the estimate \[
    \left\|\int_{\mathbb R^3\backslash \overline{O}}\chi_1(\cdot) P(\cdot,y;\lambda)\chi_2(y)q(y)\,\mathrm{d}y\right\|_{H^1(\mathbb R^3\backslash \overline{O})} \le Ce^{L(\Im \lambda)_-}|\lambda|^{2}\|q\|_{L^2(\mathbb R^3\backslash \overline{O})}.
    \] Consequently, we get \[
    \|\chi_1 R(\lambda)\chi_2\|_{H^1(\mathbb R^3\backslash \overline{O}) \to H^1(\mathbb R^3\backslash \overline{O})} \le Ce^{L(\Im \lambda)_-}|\lambda|^{2}.
    \]
  Furthermore, from standard interior elliptic regularity theorem \cite[(7.13)]{Stein}, we obtain 
  \begin{align}\label{ert}
  &\|\widetilde P(\cdot,y;\lambda)\|_{H^{n+2}(B_{R_1}\backslash \overline O)} \notag\\
  &\le C\Big(\|(\Delta+\lambda^2)\widetilde G_0(\cdot ,y;\lambda)\|_{H^n(\mathbb R^3\backslash \overline{O})}+|\lambda|^2\|\widetilde P(\cdot,y;\lambda)\|_{H^{n}(B_{R_2}\backslash \overline O)}\Big)
  \end{align}
  for all $n \in \mathbb N$ with positive constants $R_1, R_2$ such that $R_2>R_1$ and $ O \subset \subset B_{R_1}$. Then by induction, it is easy to derive that \[\|\widetilde P(\cdot,y;\lambda)\|_{H^\gamma(B_{R'}\backslash \overline O)} \le C|\lambda|^{1+\gamma}e^{L(\Im \lambda)_-}\] for any $\gamma \in \mathbb N^*,$ which implies \[
  \left\|\int_{\mathbb R^3\backslash \overline{O}}\chi_1(\cdot) P(\cdot,y;\lambda)\chi_2(y)q(y)\,\mathrm{d}y\right\|_{H^\gamma(\mathbb R^3\backslash \overline{O})} \le Ce^{L(\Im \lambda)_-}|\lambda|^{\gamma+1}\|q\|_{L^2(\mathbb R^3\backslash \overline{O})}.
  \] Therefore, we obtain resolvent estimates
   \begin{align}\label{s>01}
      \|\chi_1 R(\lambda)\chi_2\|_{H^\gamma(\mathbb R^3\backslash \overline{O}) \to H^\gamma(\mathbb R^3\backslash \overline{O})} \le Ce^{L(\Im \lambda)_-}|\lambda|^{\gamma+1} 
  \end{align} for any $\gamma \in \mathbb N^*$. Moreover,  from \eqref{ert} we have the resolvent estimates
  \begin{align}\label{s>0}
      \|\chi_1 R(\lambda)\chi_2\|_{H^\gamma(\mathbb R^3\backslash \overline{O}) \to H^{\gamma+j}(\mathbb R^3\backslash \overline{O})} \le Ce^{L(\Im \lambda)_-}|\lambda|^{\gamma+1+j}, \quad j=1, 2.
  \end{align}
  By using interpolation, the inequalities \eqref{s>01} and \eqref{s>0} hold for any $\gamma>0$, which give the resolvent estimates for $\gamma>0$.
    
 Estimates for $\gamma<0$ can be derived from the symmetry of the Green function. In fact, supposing $q \in H^\gamma(\mathbb R^3 \backslash \overline{O})$ and $\varphi \in C_c^\infty(\mathbb R^3\backslash \overline{O})$, direct calculation yields \begin{align*}
      \left\langle\chi_1 R(\lambda) \chi_2  q,  \varphi\right\rangle &=\left\langle \chi_2   q, \left\langle G_D(x,y;\lambda), \varphi \chi_1 \right\rangle_x \right\rangle_y \\ &=\left\langle  \chi_2   q , \left\langle G_D(y,x;\lambda),\varphi \chi_1  \right\rangle_x \right\rangle_y
		\\&=\left\langle   q,\chi_2 R(\lambda)\chi_1 \varphi \right\rangle.
  \end{align*} 
  It is easy to verify that the inequality \eqref{s>0} still holds if we interchange $\chi_1$ and $\chi_2$ so that \begin{align*}
      |\left\langle\chi_1 R(\lambda) \chi_2  q,  \varphi\right\rangle| &=|\left\langle   q,\chi_2 R(\lambda)\chi_1 \varphi \right\rangle| \le \|q\|_{H^\gamma}\|\chi_2 R(\lambda)\chi_1 \varphi\|_{H^{-\gamma}} \\ &\le C e^{L(\Im \lambda)_-}|\lambda|^{-\gamma+1}\|q\|_{H^\gamma(\mathbb R^3 \backslash \overline O)}\|\varphi\|_{H^{-\gamma}(\mathbb R^3 \backslash \overline O)}.
  \end{align*} By the density of $C_c^\infty(\mathbb R^3 \backslash \overline O)$ in $(H^\gamma(\mathbb R^3 \backslash \overline O))^*$, the above estimates hold for any $\varphi \in (H^\gamma(\mathbb R^3 \backslash \overline O))^*$. Therefore, we obtain 
  \[
  \|\chi_1 R(\lambda) \chi_2  q\|_{H^\gamma(\mathbb R^3 \backslash \overline O)} \le C e^{L(\Im \lambda)_-}|\lambda|^{-\gamma+1}\|q\|_{H^\gamma(\mathbb R^3 \backslash \overline O)}
  \] 
  for any real number $\gamma<0$ which completes the proof of resolvent estimates for $j=0$. The proof of the remaining cases for $j=1, 2$ follows the above arguments. The analyticity follows from Proposition \ref{nte} and the proof of \cite[Lemma 3.1]{WXZod}. The proof is complete.
\end{proof}

Denote the spatial support of $f(x, t)$ by $D$.
In the following theorem, we  show the pathwise existence and uniqueness of the solution for the stochastic scattering problem.
\begin{theorem}\label{DPT}
   The stochastic time-domain acoustic scattering problem \eqref{eqn} almost surely admits a unique solution  $u \in C^\infty(\mathbb R^+, L^2(B_R\setminus \overline{O})) \cap L^2(\mathbb R^+,L^2(B_R\setminus \overline{O}))$ with the following energy estimate \begin{equation}\label{eeee}
        \int_0^\infty \|u\|^2_{L^2(B_R\setminus \overline{O})}{\rm d}t \leq C \|\dot{W}_x\|^2_{H^{-3/2-\epsilon}(D)},
    \end{equation}
    where $C$ is a positive constant dependent on the strength $f$.
\end{theorem}
\begin{proof}
The uniqueness of the solution follows from the uniqueness result for the deterministic problem (cf. \cite{colton}).

Rewrite the Helmholtz equation in \eqref{helm} as
\[
-\Delta u_L -\lambda^2u_L = f_L\dot{W}_x
\]
where $\lambda={\rm i}s$. By Proposition \ref{ntes} and noting $\Re s=\Im \lambda$, we have that there exists $\sigma_0>0$ such that $u_L$ is a homomorphic function
in the half-plane $\Re s>\sigma_0$. Moreover, since $f(x, t)$ is smooth and compactly supported in $D\times (0, T)$,
\begin{align}\label{decays}
f_L(x, s)&=\int_0^\infty f(x, t)e^{-st}{\rm d}t\notag\\
& = \int_0^\infty \frac{1}{(-s)^m}\frac{{\rm d}^m(e^{-st})}{{\rm d}t^m}f(x, t){\rm d}t=s^{-m}\int_0^\infty e^{-st}
\partial_t^m f(x, t){\rm d}t.
\end{align}
Then we have the following estimate for $s\in \{s: \Re s>\sigma_0\}$
\begin{align}\label{est_2}
\|u_L(\cdot, s)\|_{L^2(B_R\setminus \overline{O})} &\leq C|s|^{4} \|f_L \dot W_x\|_{H^{-2}(D)} \notag\\ &\le C|s|^{4}(1+|s|^2)^{-m/2}\|\dot{W}_x\|_{H^{-3/2-\epsilon}(D)},
\end{align}
where $m$ is any positive integer, and $C$ is a generic constant dependent on $m$ and $f$.
As a consequence, by Lemma \ref{lem:a} the inverse Laplace transform of $u_L$ exists and is supported in $[0, \infty)$, which yields
a distributional solution $u$ to the acoustic wave equation in \eqref{eqn}. 

In what follows, we consider the regularity of the solution. 
Denote $u:=\mathcal L^{-1}(u_L)$. 
Noticing $s=s_1 + {\rm i}s_2$ and the fact that $u_L=\mathcal L(u)=\mathcal F_{t \to s_2}(e^{-s_1t}u)$ where $\mathcal F_{t \to s_2}$ is the Fourier transform with respect to $t$, the estimate \eqref{est_2} implies $\mathcal F^{-1}_{s_2 \to t}(u_L) \in\cap_{m=0}^{\infty}H^m(\mathbb R,L^2(B_R\setminus \overline{O}))=C^\infty(\mathbb R,L^2(B_R\setminus \overline{O}))$. Thus, we have $u = e^{s_1t} \mathcal F^{-1}_{s_2 \to t}(u_L) \in C^\infty (\mathbb R,L^2(B_R\setminus \overline{O}))$. 
Consequently, combining the smoothness and the fact that $\overline{\text{supp}\, u(x,\cdot)} \subset [0,\infty)$ yields that $u$ satisfies the initial value conditions in \eqref{eqn}. 

Now we prove the energy estimate \eqref{eeee}.
By Parseval identity we have \[
\int_0^\infty e^{-2s_1t}\mathbb\|u\|^2_{L^2(B_R\backslash \overline O)}{\rm d}t = 2\pi\int_{-\infty}^{+\infty}\|u_L\|^2_{L^2(B_R \backslash \overline O)}{\rm d}s_2.
\]
Then applying \eqref{est_2} gives
\begin{align}\label{est1}
\int_0^\infty e^{-2s_1t}\|u\|^2_{L^2(B_R \backslash \overline O)}{\rm d}t &= 2\pi\int_{-\infty}^{+\infty}{\|u_L\|^2_{L^2(B_R\backslash \overline O)}}{\rm d}s_2\notag\\
&\le C^2 \|\dot{W}\|^2_{H^{-3/2-\epsilon}(D)}\int_{-\infty}^{+\infty}|s|^8(1+|s|^2)^{-m}{\rm d}s_2.
\end{align}
Letting $s_1\to 0$ in the estimate \eqref{est1} and $m$ be sufficiently large, we obtain $u \in L^2(\mathbb R^+,L^2(B_R))$ with the energy estimate \eqref{eeee}, which completes the proof.
\end{proof}

\section{PML method for the stochastic acoustic scattering problem}\label{pml}

In this section, we propose the PML method for solving the stochastic scattering problem and establish its convergence analysis.
The key ingredient in the analysis is deriving an approximation of the acoustic equation \eqref{eqn} by replacing the spatial white noise with its piecewise constant
approximation. 

Let
\[
\dot{W}_h = \sum_{K\in\mathcal T_h}|K|^{-1/2}\xi_K\chi_K(x),
\]
where $\xi_K = \frac{1}{\sqrt{|K|}}\int_K 1{\rm d}W(x)$. Here $\{\mathcal T_h\}$ is a family of triangulation of $B_R\setminus \overline{O}$, and $h\in(0, 1)$ is the mesh size.
We have from \cite{wal} that $\{\xi_K\}_{K\in\mathcal T_h}$ is a
family of independent identically distributed normal random variables with mean
0 and variance 1.
Let $u_L^{(h)}$ be the solution to the Helmholtz equation with source term $f_{L}\dot{W}_h$, that is,
\begin{align}\label{helma}
\begin{cases}
-\Delta u_L^{(h)} - k^2u_L^{(h)}  = f_{L}\dot{W}_h,&\quad x\in\mathbb R^3\setminus \overline{O},\\
u_L^{(h)} = 0, &\quad x\in\partial O,
\end{cases}
\end{align}
where $u_L^{(h)}$ satisfies the radiation condition.
It is known that $f_{L}\dot{W}_h\in L^2(D)$. Thus,
by applying Proposition \ref{nte} we have that $u_L^{(h)}(s)$ is also holomorphic in the half-plane $\Re(s)>\sigma_0$ with the  following estimate for $s\in \{s: \Re s>\sigma_0\}$
\[
\|u_L^{(h)}(s)\|_{L^2(B_R\setminus \overline{O})}\lesssim (1+|s|)^{-1}\|\dot{W}_h\|_{L^2(D)}.
\]
As a consequence, from Lemma \ref{lem:a} there exists a unique distributional solution $u^{(h)}$ to the following acoustic wave equation whose Laplace transform is $u_L^{(h)}$
\begin{align}\label{eqn_2}
\begin{cases}
\partial^2_t u^{(h)}(x, t)-  \Delta u^{(h)}(x, t) =  f(x, t)\dot{W}_h, \quad &(x, t)\in(\mathbb R^3\setminus \overline{O})\times [0, \infty),\\
u^{(h)}(x, t) =0, \quad &(x, t)\in\partial O\times [0, \infty),\\
u^{(h)}(x, 0) = \partial_tu^{(h)}(x, 0)=0, \quad &x\in \mathbb R^3\setminus \overline{O}.
\end{cases}
\end{align}

Recalling the proof of Proposition \ref{ntes}, the Green function of the Helmholtz equation \eqref{helm} can be decomposed as follows:
\begin{align*}
G_D(x, y; {\rm i}s) = G_0(x, y; {\rm i}s) + P(x, y; {\rm i}s),
\end{align*}
where $P(x, y; {\rm i}s)$ is a Lipschitz continuous function of $x$ and $y$ in $B_R\setminus\overline O$.

The following lemma is a direct consequence of \cite[Lemma 2]{CZZ}.
\begin{lemma}\label{2.2}
It holds the estimates for $y, z\in B_R\setminus\overline{O}$
\begin{align*}
\int_{B_R\setminus\overline{O}}|G_D(x, y; {\rm i}s) - G_D(x, z; {\rm i}s)|^2{\rm d}x\leq  C|y-z|, 
\end{align*}
where $C>0$ is a positive constant.
\end{lemma}

In the following lemma, we prove an error estimate between $u_L$ and $u_L^{(h)}$. The proof is motivated by \cite[Theorem 1]{CZZ}.
It should be noticed that the explicit dependence on the strength function $f_L$ is derived in our result, which plays an important role in the subsequent analysis.

\begin{lemma}
Let $u_L$ and $u_L^{(h)}$ be the solutions to the Helmholtz equations \eqref{helm} and \eqref{helma}, respectively.
It holds the estimate
\begin{align}\label{3.2}
\mathbb E[\|u_L - u_L^{(h)}\|^2_{L^2(B_R\setminus \overline{O})}]\leq
C\|f_L(\cdot, s)\|_{C^1(B_R\setminus \overline{O})}h,
\end{align}
where $C$ is a positive constant independent of $u_L$ and $h$.
\end{lemma}

\begin{proof}
Given a function $g$, denote
\begin{align*}
(\mathcal G g)(x) = \int_{B_R\setminus\overline{O}} G_D(x, y; {\rm i}s)g(y){\rm d}y.
\end{align*}
By subtracting \eqref{helm} from \eqref{helma} we obtain
\begin{align*}\label{subt}
u - u_h  = E_h,
\end{align*}
where
\[
E_h =\mathcal{G}(f_L\dot{W}_x)-\mathcal{G}(f_L\dot{W}_h).
\]
From Ito isometry we obtain
\begin{align*}
&\mathbb E[\|\mathcal{G}(f_L\dot{W}_x) - \mathcal{G}(f_L\dot{W}_h)\|^2_{L^2(B_R\setminus \overline{O})}]\\
&=\mathbb E\Big[\int_{B_R\setminus \overline{O}}\Big(\int_{B_R\setminus \overline{O}}f_L(y, s)G_D(x, y; {\rm i}s){\rm d}W(y) - \int_{B_R\setminus \overline{O}}f_L(y, s)G_D(x, y; {\rm i}s){\rm d}W_h(y)\Big)^2{\rm d}x\Big]\\
&=\mathbb E\Big[\int_{B_R\setminus \overline{O}}\Big(\sum_{K\in\mathcal{T}_h}\int_{K}f_L(y, s)G_D(x, y; {\rm i}s){\rm d}W(y)\\
&\quad - |K|^{-1}\sum_{K\in\mathcal{T}_h}\int_{K} f_L(z, s)G_D(x, z; {\rm i}s){\rm d}z\int_K1{\rm d}W(y)\Big)^2{\rm d}x\Big]\\
&=\mathbb E\Big[\int_{B_R\setminus \overline{O}}\Big(\sum_{K\in\mathcal T_h}\int_{K}\int_K|K|^{-1}(f_L(y, s)G_D(x, y; {\rm i}s)-f_L(z, s)G_D(x, z; {\rm i}s)){\rm d}z{\rm d}W(y)\Big)^2{\rm d}x\Big]\\
&=\int_{B_R\setminus \overline{O}}\Big(\sum_{K\in\mathcal T_h}\int_{K}\Big( |K|^{-1}\int_K(f_L(y, s)G_D(x, y; {\rm i}s)-f_L(z, s)G_D(x, z; {\rm i}s)){\rm d}z\Big)^2{\rm d}y\Big){\rm d}x,
\end{align*}
which yields from the H\"older inequality
\begin{align}\label{lip}
&\mathbb E[\|\mathcal{G}(f_L\dot{W}_x) - \mathcal{G}(f_L\dot{W}_h)\|^2_{L^2(B_R\setminus \overline{O})}]\notag\\
&\leq \int_{B_R\setminus \overline{O}}\Big(\sum_{K\in\mathcal T_h}|K|^{-1}\int_{K}\int_K|f_L(y, s)G_D(x, y; {\rm i}s)-f_L(z, s)G_D(x, z; {\rm i}s)|^2{\rm d}z{\rm d}y\Big){\rm d}x.
\end{align}
Since
\begin{align*}
&|f_L(y, s)G_D(x, y; {\rm i}s)-f_L(z, s)G_D(x, z; {\rm i}s)|\\
&\leq f_L(y, s)|G_D(x, y; {\rm i}s)-G_D(x, z; {\rm i}s)| + |f_L(y, s) - f_L(z, s)||G_D(x, z; {\rm i}s)|,
\end{align*}
we obtain from \eqref{lip} and Lemma \ref{2.2} that
\begin{align*}
&\mathbb E[\|\mathcal{G}(f_L\dot{W}_x) - \mathcal{G}(f_L\dot{W}_h)\|^2_{L^2(B_R\setminus \overline{O})}]\notag\\
&\leq C\Big(\|f_L(\cdot, s)\|_{C^1(B_R\setminus \overline{O})}\int_{B_R\setminus \overline{O}}\Big(\sum_{K\in\mathcal T_h}|K|^{-1}\int_{K}\int_K\frac{|y-z|^2}{|x-z|^2}{\rm d}z{\rm d}y\Big){\rm d}x \notag\\
&\quad+ \int_{B_R\setminus \overline{O}}\Big(\|f_L(\cdot, s)\|_{C(B_R\setminus \overline{O})}\sum_{K\in\mathcal T_h}|K|^{-1}\int_{K}\int_K|G_D(x, y; {\rm i}s)-G_D(x, z; {\rm i}s)|^2{\rm d}z{\rm d}y\Big){\rm d}x\Big)\notag\\
&\leq C\|f_L(\cdot, s)\|_{C^1(B_R\setminus \overline{O})}\sum_{K\in\mathcal T_h}|K|^{-1}\int_{B_R\setminus \overline{O}}\int_{K}\int_K\frac{|y-z|^2}{|x-z|^2}{\rm d}z{\rm d}y{\rm d}x \\
&\quad+ C\|f_L(\cdot, s)\|_{C(B_R\setminus \overline{O})}\sum_{K\in\mathcal T_h}|K|^{-1}\int_{B_R\setminus \overline{O}}\int_{K}\int_K|G_D(x, y; {\rm i}s)-G_D(x, z; {\rm i}s)|^2{\rm d}z{\rm d}y{\rm d}x\\
&\leq C\|f_L(\cdot, s)\|_{C^1(B_R\setminus \overline{O})}h^2   \\
&\quad +C\|f_L(\cdot, s)\|_{C(B_R\setminus \overline{O})}\sum_{K\in\mathcal T_h}|K|^{-1}\int_{B_R\setminus \overline{O}}\int_{K}\int_K|G_D(x, y; {\rm i}s)-G_D(x, z; {\rm i}s)|^2{\rm d}z{\rm d}y{\rm d}x.
\end{align*}
Therefore, the estimates in \eqref{3.2} follow from the above inequality and Lemma \ref{2.2}.
\end{proof}

In the following theorem, we establish an error estimate of $u - u^{(h)}$. As a consequence, to solve the scattering problem \eqref{eqn}, one might compute $u_h$ by using the classical
computational methods such as PML  and absorbing boundary condition methods,  since the driven source in the wave equation \eqref{eqn_2} is now regular in $C^\infty([0, T], L^2(B_R\setminus\overline{O}))$. 
\begin{theorem}\label{approxi}
Let $u$ and $u^{(h)}$ be the solutions to the acoustic wave equations \eqref{eqn} and \eqref{eqn_2}, respectively.
It holds the estimate
\begin{align*}
\mathbb E\Big[\|u-u^{(h)}\|^2_{L^2(0, T; L^2(B_R\setminus \overline{O}))}\Big]\leq 
Ch,
\end{align*}
where $C$ is a positive constant independent of $u$ and $h$.
\end{theorem}

\begin{proof}
Notice that
\[
u_L - u_L^{(h)} = \mathcal{L}(u - u^{(h)}) = \mathcal{F}(e^{-s_1t}u),
\]
where $\mathcal{F}$ is the Fourier transform in $s_2$. By Parseval identity we have
\[
\int_0^\infty e^{-2s_1t}\mathbb\|u-u_h\|^2_{L^2(B_R\setminus \overline{O})}{\rm d}t = 2\pi\int_{-\infty}^{+\infty}\mathbb \|u_L - u_L^{(h)}\|^2_{L^2(B_R\setminus \overline{O})}{\rm d}s_2.
\]
Therefore, by taking expectation and in view of the estimates \eqref{3.2} and \eqref{decays}, we obtain
\begin{align}\label{est}
&\int_0^\infty e^{-2s_1t}\mathbb E[\|u-u_h\|^2_{L^2(B_R\setminus \overline{O})}]{\rm d}t \notag\\
&= 2\pi\int_{-\infty}^{+\infty}\mathbb E[\|u_L - u_L^{(h)}\|^2_{L^2(B_R\setminus \overline{O})}]{\rm d}s_2\notag\\
&\leq 
Ch\int_{-\infty}^{+\infty}(1+|s|^2)^{-m/2}{\rm d}s_2,
\end{align}
where $m$ can be any positive integer.
Letting $s_1\to 0$ in the estimate \eqref{est} and $m$ be sufficiently large, we obtain
\begin{align}\label{est2}
\int_0^\infty \mathbb E[\|u-u_h\|^2_{L^2(B_R\setminus \overline{O})}]{\rm d}t \leq Ch.
\end{align}
In view of \eqref{est2} we obtain for any $T>0$ that
\begin{align*}
\mathbb E\Big[\|u-u_h\|^2_{L^2(0, T; L^2(B_R\setminus \overline{O}))}\Big]=\int_0^T \mathbb E[\|u-u_h\|^2_{L^2(B_R\setminus \overline{O})}]{\rm d}t 
\leq 
Ch.
\end{align*}
The proof is complete.
\end{proof}

Based on Theorem \ref{approxi}, we propose a PML method and derive the convergence analysis. Define $B_\rho:= \{x \in \mathbb{R}^3 : |x| < \rho\}$
with $\rho>R$. Introduce the PML layer $\Omega_{PML} = B_\rho\setminus\overline{B_R} =  \{x \in \mathbb{R}^3 : R<|x| < \rho\}$ with thickness
$d=\rho-R$ which surrounds the bounded domain $B_R\setminus\overline{O}$. Let $s_1>0$ be an arbitrarily fixed parameter which can be
considered as the real part of the Laplace transform variable.
Define the PML medium property
as follows
\[
\alpha(r) = 1 + s_1^{-1}\sigma(r),\quad r=|x|,
\]
where
\begin{align*}
\sigma(r)=
\begin{cases}
0,  &\quad 0\leq r\leq R, \\[5pt]
\sigma_0\Big(\frac{r-R}{\rho-R}\Big)^m, &\quad R\leq r\leq\rho,\\[5pt]
\sigma_0, &\quad 0\leq r\leq\infty.
\end{cases}
\end{align*}
Here $\sigma_0>0$ is a positive constant and $m\geq 1$ is an integer. 
Following the real coordinate stretching technique to derive the PML equation, we introduce the real stretched radius
\[
\tilde{r} = \int_0^r\alpha(\tau){\rm d}\tau = r\beta,
\]
where
\[
\beta(r)=\frac{1}{r}\int_0^r\alpha(\tau){\rm d}\tau.
\]
The truncated time-domain PML problem is formulated as follows
\begin{align}\label{PMLS}
\begin{cases}
\alpha\beta^2\partial_{tt}u^{(h)}_P - \nabla\cdot(A\nabla u^{(h)}_P) = f(x, t)\dot{W}_h,&\quad (x, t)\in (B_\rho\setminus\overline{O})\times (0, T],\\
u^{(h)}_P=0, &\quad (x, t)\in \partial O\times (0, T],\\
u^{(h)}_P=0, &\quad (x, t)\in \partial B_\rho\times (0, T],\\
u^{(h)}_P(x, 0) = \partial_t U_p(x, 0) = 0, &\quad x\in B_\rho\setminus\overline{O},
\end{cases}
\end{align}
where $A = {\rm diag}(\beta^2/\alpha, \alpha, \alpha)$ is the diagonal matrix under the polar coordinates. 

The following convergence analysis of the PML method was proved in \cite{lv}.

\begin{proposition}\label{lv}
Let $u^{(h)}$ and $u^{(h)}_P$ be the solutions to the wave equations \eqref{eqn_2} and \eqref{PMLS}, respectively, and $s_1 = 1/T$.
Then the following estimate holds:
\[
\|u^{(h)} - u^{(h)}_P\|_{L^\infty(0, T; H^1(B_R\setminus\overline{O}) )}\lesssim T^{7/2}d(1+\sigma_0T)^{11}e^{-\sigma_0d/2}\|f(x, t)\dot{W}_h\|_{H^6(0, T; L^2(B_R\setminus\overline{O}))}.
\]
\end{proposition}

Noting $\mathbb E[\|\dot{W}_h\|^2_{L^2(D)}]\sim h^{-2}$ (cf. \cite[Lemma 1]{CZZ}),
as a direct consequence of Proposition \ref{lv} by taking expectation, we obtain the following estimate in the computational domain.
\begin{lemma}\label{exlv}
It holds that 
\[
\mathbb E[\|u^{(h)}-u_P^{(h)}\|_{L^\infty(0, T; L^2(B_R\setminus\overline{O}) )}]\leq Ch^{-2}T^{7/2}d(1+\sigma_0T)^{11}e^{-\sigma_0d/2},
\] where  $C>0$  is a  constant independent of  $T, \sigma_0$,  $d$ and $h$.
\end{lemma}

With the help of Proposition \ref{approxi} and Lemma \ref{exlv}, we establish the convergence of the time-domain PML method. 
\begin{theorem}\label{main}
Let $u$ and $u_P^{(h)}$ denote the solutions to the original and PML-truncated acoustic wave problems \eqref{eqn} and \eqref{PMLS}, respectively. Then
\begin{align*}
\mathbb E[\|u - u_P^{(h)}\|^2_{L^2(0, T; L^2(B_R\setminus\overline{O}))}]\leq Ch + Ch^{-2}T^{7/2}d(1+\sigma_0T)^{11}e^{-\sigma_0d/2},
\end{align*}
where  $C>0$  is a  constant independent of  $T, \sigma_0$,  $d$ and $h$.
\end{theorem}

\begin{proof}
By the triangulation inequality, we decompose the total error as
 \[
\|u - u_P^{(h)}\|^2_{L^2(0, T; L^2(B_R))}\leq \|u - u^{(h)}\|^2_{L^2(0, T; L^2(B_R))} + \|u^{(h)} - u_P^{(h)}\|^2_{L^2(0, T; L^2(B_R))}.
\]
Taking expectations and applying Theorem \ref{approxi} and Lemma  \ref{exlv} yield the desired result.
\end{proof}

The convergence analysis consists of two parts: the first part is the error estimate in Lemma  \ref{exlv}, and the second part is from computing $u^{(h)}$
by its PML solution $u^{(h)}_P$.  It should be noticed that as the mesh size $h$ becomes smaller, the first part will decrease while the second part 
will increase due to the existence of the parameter $h^{-2}$. Therefore, to obtain a desired error estimate, one may 
choose $h$ to be sufficiently small in the first part, and then fix $h$ and adjust the parameters $\sigma_0$ and $d$ in the second part.

\section{Conclusion}

In this paper, we consider the stochastic acoustic scattering problem driven by an additive Gaussian noise which is spatially white.
The pathwise well-posedness of the scattering problem is proved by applying anabstract Laplace transform inversion theorem.
To address the low regularity of the source, we employ scattering theory to study the Helmholtz resolvent defined on rough fields.
Based on the piecewise constant approximation of the white noise, we construct an approximation solution and derive an error estimate
in $L^2$-norm. As a consequence, we present a PML method and established its convergence analysis. 
The convergence of the PML method exhibits explicit dependence on the thickness and medium property of the PML layer, as well as the spatial mesh size of the approximation of the white noise. 
Currently, we are developing numerical methods to solve the PML problem.
We hope to report the progress of the ongoing study in a forthcoming publication.

%

\end{document}